\documentclass{amsart}
\usepackage[headings]{fullpage}
\usepackage{lmodern}
\usepackage[T1]{fontenc}
\usepackage{amsmath}
\usepackage{amsthm}
\usepackage{amssymb}
\usepackage{amsfonts}
\usepackage{latexsym}
\usepackage{paralist, enumitem}
\usepackage{color}
\usepackage[colorlinks=true,pagebackref,hyperindex,citecolor=NavyBlue,linkcolor=Mahogany]{hyperref}
\usepackage[dvipsnames]{xcolor}

\makeatletter

\@addtoreset{equation}{section}
\def\theequation{\thesection.\@arabic \c@equation}
\def\@citecolor{blue}
\def\@urlcolor{blue}
\def\@linkcolor{blue}
\def\theenumi{\@roman\c@enumi}

\makeatother

\theoremstyle{plain}
\newtheorem{theorem}[equation]{Theorem}
\newtheorem{lemma}[equation]{Lemma}
\newtheorem{corollary}[equation]{Corollary}
\newtheorem{proposition}[equation]{Proposition}
\newtheorem*{theorem*}{Theorem}
\newtheorem*{claim*}{Claim}

\theoremstyle{definition}
\newtheorem{definition}[equation]{Definition}
\newtheorem{example}[equation]{Example}

\newtheorem{remark}[equation]{Remark}

\newtheorem{discussion}[equation]{Discussion}

\DeclareMathOperator{\Char}{char}

\DeclareMathOperator{\tor}{Tor}

\def\frk{\mathfrak}               

\def\mm{{\frk m}}

\title{A hyperplane restriction theorem and applications to reductions of ideals}
\author[G.~Caviglia]{Giulio Caviglia}
\address{Department of Mathematics, Purdue University, West Lafayette IN
47907, USA}
\email{gcavigli@purdue.edu}

\thanks{The work of author is supported by a grant from the Simons Foundation (41000748, G.C.)}


\begin{document}

\begin{abstract}
Green's general hyperplane restriction theorem gives a sharp upper bound for the Hilbert function of a standard graded algebra over and infinite field $K$  modulo a general linear form. We strengthen Green's result by showing that the linear forms that do not satisfy such estimate belong to a finite union of proper linear spaces. As an application we give a method to derive variations of the Eakin-Sathaye theorem on reductions. In particular, we recover and extend results by O'Carroll on the Eakin-Sathaye theorem for complete and joint reductions.
\end{abstract} 
 
\maketitle

\section*{Introduction}

Let $R= \bigoplus_{d\in \mathbb{N}} R_d $
be a standard graded algebra over an infinite field $K$. A well-known result of Green \cite{Gr1} provides an upper bound for the dimension of the graded component of degree $d$ of  $R/lR$, where $l$ is a linear form, in terms of the dimension of the graded component of degree $d$ of $R.$ Such a bound is satisfied generically, in the sense that it holds for any linear form in a certain  non-empty Zariski open set $U\subseteq \mathbb{A}(R_1).$ 
This estimate, known as \emph{general hyperplane restriction theorem} is one of the most useful results in the study of Hilbert functions of graded algebras. It plays a central role in modern proofs of many classical theorems on Hilbert functions; Macaulay's characterization of all the possible Hilbert functions of standard graded algebras, Gotzmann's persistence theorem and Gotzmann's regularity theorem \cite{Go} are among those (see for instance \cite{Gr1}, \cite{BH} Section 4.2 and 4.3, and \cite{Gr2} Section 3). 

One of the main results of this paper is a strengthening of the general hyperplane restriction theorem. We show, in Theorem \ref{MAIN},  that the Zariski open set of linear  forms satisfying Green's bound contains the complement of a finite union of proper linear subspaces. 
The technical aspect of this result is discussed in Section \ref{Section2} where a more general statement, Theorem \ref{NewGreenTHM}, is presented.

The central part of the proof of Theorem \ref{NewGreenTHM} follows closely Green's original paper  \cite{Gr2}, with the main difference that we underline the key-properties (see Definition \ref{KEY})  needed in order to build up the inductive steps of the argument.  Even though these properties are rather technical, the most common situations in which they are satisfied are quite simple, and allows us to derive Theorem \ref{MAIN}. 

The goal of the first half of this paper, as mentioned above, is to provide a method to substitute the genericity condition for the linear form with a weaker assumption. 
The following example will perhaps give the reader some motivation of why a modification of Green's result is desirable. Assume for instance that the standard graded algebra $R$ in Green's result is a quotient of  the following toric algebra: 
\[S=K[X_iY_j \vert 1\leq i\leq n_1, 1\leq j \leq n_2]\cong K[T_1,\dots,T_{n_1n_2}]/I.\]     
For such an algebra it is reasonable to expect that a linear form of $R$ which is the image of the product of a general linear form in the $X_i$'s and a general linear form in the $Y_j$'s may satisfy Green's bound. We will show, as a simple consequence of Theorem \ref{NewGreenTHM},  that this is in fact the case, even though such an element is not general since linear forms of this kind  belong to a non-trivial Zariski closed set. 

 In the second half of the paper we provide an application of Theorem \ref {MAIN} to the theory of reductions of ideals in local rings. Let $(A,\mm)$ be a local ring  and let $I\subset A$ be an ideal. A \emph{reduction} of $I$ is an ideal $J\subset I$ such that  $I^{n+1}=JI^n$ for some non-negative integer $n.$ The notion of reduction, which  was introduced by Northcott and Rees in \cite{NR}, has been widely used in many areas, including  multiplicity theory and the theory of blow-up rings. We refer the reader to the dedicated chapter in \cite{HuSw}.
 
 Green's general hyperplane restriction theorem  can be employed (see \cite{Ca1}, or \cite{HuSw} section 8.6) to give a short proof of the Eakin-Sathaye theorem (see  \cite{ES}, \cite{Sa} and also \cite{HT}) a well known result on reduction of ideals.  Precisely, when $\vert R/ \mm \vert =\infty$, the main theorem of \cite{ES} says that for an integer $p$ large enough so that the number of generators of $(I^i)$ is smaller than $\binom{i+p}{p}$, there exists a reduction 
$(h_1,\dots,h_p)$ of $I$ such that $I^i=(h_1,\dots,h_p)I^{i-1}.$  

The result of \cite{ES} has been generalized by O'Carroll \cite{O} (see also \cite{BE} and \cite{GoSuVe}) to the case of complete and joint reductions in the sense of Rees. It is worth to note that the elements $h_1,\dots,h_p,$ can be chosen to correspond to general linear forms of the fiber cone ring $R=\bigoplus_{i\geq 0} I^i/ \mm I^i.$  Theorem \ref{NewGreenTHM}, by strengthening Green's general hyperplane theorem, has the direct consequence of allowing for variations of the Eakin-Sathaye theorem. Specifically, the weakening of the hypothesis on the general linear forms allows us to recover and extend O'Carroll results to a broader range of situations (see Section 2).

The content of this paper, with the exemption of the main result Theorem \ref{MAIN}, is taken form Chapter 8 of the author's Ph.D thesis. We are thankful to David Eisenbud for the comments he provided on an earlier version of this manuscript and to Alessandro De Stefani for important suggestions regarding Section 1, in particular on how to relax assumptions on the field $K.$

\section{A Hyperplane Restriction Theorem}\label{Section2}

Let $R$ be a standard graded algebra over an infinite field $K$. We can
write $R$ as $A/I$, where $A=K[X_1,\dots,X_n]$ and $I$ is a homogeneous ideal. 
In the following, when we say that a property $(P)$ is satisfied by $r$ general linear forms of $R$ we mean that there exists a non-empty Zariski open set $U\subseteq \mathbb{A}(R_1)^r$ such that any $r$-tuple in $U$ consists of $r$ linear forms satisfying $(P).$ 

Recall that given  a positive integer $d$,
any other non-negative integer $c$ can then be uniquely expressed in term of $d$ as 
$c=\binom{k_d}{d}+\binom{k_{d-1}}{d-1}+\dots + \binom{k_1}{1},$
where the $k_i$'s are non-negative and strictly decreasing, i.e. 
$k_d>k_{d-1}>\dots>k_1 \geq 0$. This way of writing $c$ is called the
$d$'th \emph{Macaulay representation} of $c,$ and the $k_i$'s are
called the $d$'th \emph{Macaulay coefficients} of $c.$
The integer $c_{\langle d \rangle }$ is defined to be
$c_{\langle d \rangle }=\binom{k_d-1}{d}+\binom{k_{d-1}-1}{d-1}+\dots + \binom{k_1-1}{1}.$ 

Mark Green proved the following:
\begin{theorem*}[General Hyperplane Restriction Theorem]\label{Green1}
Let $R$ be a standard graded algebra over an infinite field $K$, let $d$ be a degree and let $l$ be a general linear form of $R.$ Then 
\begin{equation}\label{grr}
\dim_K (R/lR)_d\leq (\dim_K R_d)_{\langle d \rangle }. 
\end{equation}
\end{theorem*}

The above result was first proved in \cite{Gr1} with no assumption on the characteristic of the base
field $K.$  For the case $\Char(K)=0$, a more combinatorial and perhaps simpler proof be found in \cite{Gr2}; it uses generic initial ideals and  it gives at the same time a proof of Macaulay's estimate on Hilbert functions.  Despite the extensive literature, I am not aware of any argument based solely on generic initial ideals that would derive \eqref{grr} in any characteristic.  

It is important to recall that the numerical bound \eqref{grr} can be also interpreted in the following way: let $A=K[X_1,\dots,X_n]$ and let $I\subset A$  be a homogeneous ideal. Define  $I^{\text{lex}}\subset A$ to be the unique lex-segment ideal with the same Hilbert function as $I.$ Let $c$ be the dimension, as a $K$-vector space, of $(A/I)_d.$ By definition we also have that  $\dim_K (A/I^{\text{lex}})_d=c.$ It is possible to show that $\dim_K (A/(I^{\text{lex}}+(X_n))_d=c_{\langle d \rangle }.$ For any degree $d$, the general hyperplane restriction theorem  is equivalent to the statement that if $l$ is general then
\begin{equation}\label{grr1}
\dim_K (A/I+(l))_d \leq \dim_K (A/I ^{\text{lex}}  +(X_n))_d.
\end{equation}
The inequality \eqref{grr1} has been generalized by several authors. Aldo Conca \cite{Co} has proved 
that when characteristic of $K$ is zero and $l_1,\dots, l_r$ are general linear forms one has
\begin{equation}
\dim_K \tor_i(A/I,A/(l_n,\dots,l_r))_d \leq  \dim_K \tor_i(A/I^{\text{lex}},A/(X_n,\dots,X_r))_d \quad \text{ for all }i \text{ and } d.
\end {equation}
Conca's result, when $i=0$, corresponds to the characteristic zero case of Green's theorem. 
Unfortunately, the method used in \cite{Co} requires the use of general initial ideals which are also stongly stable and, as mentioned above, this puts some restriction on the possible characteristic of $K$. 
In a different direction Herzog and Popescu \cite{HP} and Gasharov \cite{Ga} extended \eqref{grr1} by showing that the next inequality holds  for a general form  $f$ of $A$ of degree $c$
 \begin{equation}\label{grr2}
\dim_K (A/I+(f))_d \leq \dim_K (A/I ^{\text{lex}}  +(X_n^c))_d. 
\end{equation}
Further generalizations in this direction can be found in \cite{CaMu}.

The goal of the remaining part of this section is to show how the assumption of generality on the linear form satisfying $\eqref{grr}$, or equivalently \eqref{grr1}, can be relaxed. We first introduce some notation.

Let $R$ be a standard graded algebra, we denote by $\frk m$ its homogeneous maximal ideal and let ${\bf l}= l_1,\dots,l_r$ be a sequence of linear forms.
For every $0\leq m\leq r$ and every sequence $\mathbf o= o_1,\dots,o_m$  of $m=\vert \mathbf o \vert $ letters in the set $\{c,s\},$ we construct an ideal $I_{{\mathbf l}, \mathbf o}$  by recursively considering colons and sums of the above linear forms.

When  $\mathbf o$ is empty we let $ I_{\mathbf l, \mathbf o}=(0)$ and when  $\mathbf o$ equals $c$ or $s$ we set $I_{\mathbf l, \mathbf o}$ to be $(0):l_1$ and $(l_1)$ respectively. In general if $\mathbf {\bar o}=o_1,\dots, o_i$ and $\mathbf o= {\mathbf {\bar o}},o_{i+1}$ we set $I_{\mathbf l , \mathbf o}$ to be $I_{\mathbf l , \mathbf {\bar o}}:l_{i+1}$ 
or $I_{\mathbf l , \mathbf {\bar o}}+(l_{i+1})$ depending whether $o_{i+1}$ is $c$ or $s.$ We also let $\vert {\mathbf o} \vert_c$
 and $\vert \mathbf o \vert _s$ be, respectively, the number of letters $c$ and of letters $s$ in $\mathbf o$. For simplicity, whenever it is clear from the context what the sequence $\mathbf l$ is, we will just write $I_{\mathbf o}$ instead of $I_{\mathbf l , \mathbf o}.$

\begin{definition}\label{KEY}
We say that $l_1,\dots,l_r$ satisfy property \textbf{(Gr,$d$)}, meaning they are suitable for a hyperplane restriction theorem in  degree $d>0$, if for every sequence $\mathbf o = o_1,\dots,o_i$ with $0\leq i\leq r$, the following hold:

\begin{enumerate}
\item \label{11}
If $\frk m \not \subseteq I_{\mathbf o}$ and $i<r$, then $l_{i+1}\not \in  I_{\mathbf o}$,
\item \label{22}
If $i=r$ and $\vert \mathbf o \vert_c<d$ then $\frk m \subseteq I_{\mathbf o}$,  
\item \label{33}
If $i\leq r-2$ then $\dim_K (I_{\mathbf {o},c,s})_{d- \vert \mathbf o \vert _c -1} \leq \dim_K (I_{\mathbf {o},s,c})_{d- \vert \mathbf o \vert _c -1}$.
  \end{enumerate}
\end{definition}

 \begin{remark}\label{Remark on (3)Green}
Let $n=\dim_K R_1.$ 
If property \eqref{11}  holds, then property \eqref{22} is automatically satisfied whenever $r\geq n+d-1$ and clearly in this case $l_1,\dots, l_r$ generate $\frk m.$
Property \eqref{33} is implied by the next stronger condition. 

\begin{enumerate} \setcounter{enumi}{3}
\item \label{44} If $i\leq r-2$ then  $\dim_K(I_{\mathbf o,c,s})_{d- \vert \mathbf o \vert _c -1} =\dim_K((I_{\mathbf o}:l_{i+2})+l_{i+1})_{ d- \vert \mathbf o \vert _c -1}.$ 
\end{enumerate}
To see that this is the case, notice that $(I_{\mathbf o}:l_{i+2})+(l_{i+1})\subseteq I_{\mathbf o,s,c}$  and thus  \eqref{44} gives:
$\dim_K (I_{\mathbf {o},c,s})_{d- \vert \mathbf o \vert _c -1} = \dim_K((I_{\mathbf o}:l_{i+2})+l_{i+1})_{ d- \vert \mathbf o \vert_c -1} \leq  \dim_K (I_{\mathbf {o},s,c})_{d- \vert \mathbf o \vert_c -1}.$
\end{remark}

At a first sight the properties \textbf{(Gr,$d$)} may not seem easy to verify. However, there are  several examples for which it is not hard to find linear forms satisfying them. For instance let $r=n+d-1$, and assume that the linear forms $l_1,\dots,l_r$ span $\frk m$ and that for each $\mathbf o$ the Hilbert functions of the ideals $I_{\mathbf o}$ at the degrees between $0$ and $d-\vert \mathbf o \vert_c-1$ do not depend on the order of the $l_1,\dots,l_r.$ This implies immediately \eqref{11} and \eqref{44}, hence $l_1,\dots, l_r$ satisfy $\textbf{(Gr,$d$)}.$ 
We summarize the above considerations in the following lemma.

\begin{lemma} \label{easy} Let $n=\dim_K R_1,$ $d$ be  a degree, $r\geq n+d-1$,  and $l_1,\dots , l_r$ linear forms of $R$ generating $\frk m.$ If for every sequence $\mathbf o=c_1,\dots, c_i$ with $0\leq i\leq r$, and for every $j\in \{0,\dots, d-\vert \mathbf o \vert_c -1 \}$ 
we have that  $dim_K(I_{\mathbf o})_j$ is independent of the order of the $l_s$'s then  $l_1,\dots, l_r$ satisfy \textbf{(Gr,$d$)}. 
\end{lemma}
 
We complete our discussion of the properties \textbf{(Gr,$d$)} with the following two results, which show a simple case in which the assumptions of Lemma \ref{easy} are satisfied. The experts will be able to see why the conclusion of Proposition \ref{constant} and Corollary \ref{order} are straightforward. We include for the sake of the exposition some concise explanations. 


\begin{proposition}\label{constant} 
Let  $R=K[X_1,\dots,X_n]/J$ be a standard graded algebra over an infinite field $K$. Let $V\subset \mathbb{A}(R_1)$ be an irreducible variety and assume that $V^r = V\times \cdots\times V$ $r$ times is irreducible as well. Then for every $r$-linear forms of $R$ that are general points of $V$, and for every sequence ${\bf o}=o_1,\dots,o_i$ with
 $0\leq i\leq r$, the Hilbert function of $I_{\bf o}$ is well defined, equivalently there exists a non-empty Zariski open subset of $V^r$ on which the Hilbert function of $I_{\bf o}$ is constant.
\end{proposition}
\begin{proof} 
We consider the coordinate ring of $V,$ say $S_{V}=K[Y_1,\dots,Y_n]/I_V$ and more generally the coordinates ring $S_{V^r}=K[Y_{1,1}\dots,Y_{1,n}, \dots, Y_{r,1}\dots,Y_{r,n}]/I_{V^{r}}$ of $V^r.$ By assumption $V$ and $V^r$ are irreducible, hence their defining ideals $I_V$ and $I_{V^r}$  are prime. Let $\mathbb K$ be the fraction field of $S_{V^r}.$ For every $a,b$ we denote with $y_{a,b}$ the image in $\mathbb K$ of $Y_{a,b}.$ Let $\mathbf{{\tilde l}}$ be image in $R\otimes_K \mathbb K$ of the sequence of linear forms  $\tilde{l_1},\dots,\tilde{l_r}$ of  $\mathbb  K[X_1,\dots,X_n],$    where, for $a=1,\dots,r$, we have set $l_a=y_{a,1}X_1+\cdots+y_{a,n}X_n.$ There is a standard way, using Gr\"obner bases, to explicitly compute a reduced Gr\"obner basis, for instance with respect to the reverse lexicographic order, of the pre-image in  $\mathbb  K[X_1,\dots,X_n]$ of the ideal $I_{\tilde {\mathbf l}, \mathbf o} \subseteq R \otimes_{K}\mathbb K.$ We consider all the non-zero coefficients $\alpha_1,\dots,\alpha_q$ of all the monomials appearing in all the polynomials involved in such computation. We let $U_j$ be the non-empty Zariski open set of $V^r$ where the rational function $\alpha_j$ does not vanish. The desired Zariski open subset is $U=\cap_j^q U_j$ which is not empty since $V^r$ is irreducible. By construction, for any point in $U$ the initial ideal of $I_{\bf o}$ is always the same and therefore the Hilbert function  
is constant as well.
\end{proof}

As a consequence of the above proof we get the following fact.

\begin{corollary}\label{order} With the same assumption as Proposition \ref{constant},
for every $r$-linear forms of $R$ that are general points of $V$, and for every sequence ${\bf o}=o_1,\dots,o_i$ with $0\leq i \leq r$ the Hilbert function of $I_{\bf o}$ is independent of the order of the linear forms.
\end{corollary} 
 \begin{proof} In the proof of Proposition \ref{constant} we constructed a non-empty Zariski open set $U$ by means of  a Gr\"obner basis computation of the pre-image in $\mathbb  K[X_1,\dots,X_n]$ of the ideal $I_{\tilde {\mathbf l}, \mathbf o} \subseteq R \otimes_{K}\mathbb K.$ By reordering the sequence of linear forms $\tilde {\mathbf l}$, for every given permutation $\sigma \in \mathcal S_n$ we get a non-empty Zariski open set $U_\sigma$ where the Hilbert function of $ I_{\sigma (\mathbf l), \mathbf o}$ in independent of the choice of $\mathbf l \in U_{\sigma}$.  By assumption $V^r$ is irreducible, hence $\bar U=\cap_{\sigma \in \mathcal S} U_{\sigma} $ is a non-empty Zariski open set. For every point  $\mathbf l \in \bar U$  the Hilbert function of $I_{\bf o}$ is independent of the order of the linear forms in $\mathbf l.$
 \end{proof}

\begin{remark}\label{irreducible} Regarding the hypotheses of the above results, notice that when $V\subset \mathbb{A}(R_1)$ is an irreducible variety, it is possible to conclude that $V^r$ is irreducible as well provided that $K$ is algebraically closed or alternatively, with no assumption on $K$, provided that the defining ideal of $V$ is homogeneous.
\end{remark}

In the following examples one can apply Lemma \ref{easy} to derive the property $\textbf{(Gr,$d$)}$.
\begin{example}
Let $R=K[X,Y,Z]/I$ be a standard graded algebra over an algebraically closed field $K$. Then for a general choice of $\lambda_1,\dots,\lambda_{d+2} \in  K$ the linear forms $l_i=X+\lambda_iY+\lambda_i^2Z$ satisfy \textbf{(Gr,$d$)}.
\end{example}

\begin{example}\label{examples of (Gr,$d$)}
Let $R$ be a standard graded algebra, $\dim_K R_1=n$ and assume $\vert K \vert =\infty.$ The $r$ linear forms, with $r\geq d+n-1$, in each of the cases below satisfy \textbf{(Gr,$d$)}. These examples are relevant to the next section and will correspond to variations of the  Eakin-Sathaye theorem. All the homomorphisms mentioned below are assumed to send the monomials generating each sub-algebra to forms of $R_1.$
\begin{itemize}
\item[(A)] With no further assumptions on $R$ the $r$ linear forms can be chosen to be general.
\item[(B)] Assume $R$ to be the homomorphic image of the Segre ring: 
\[S=K[X_{1,i_1}\cdot X_{2.i_2}\cdots X_{s,i_s} \vert 1\leq i_1\leq n_1, \dots, 1\leq i_s\leq n_s ].\] Then the $r$ linear forms can be chosen to be the images of $l_1\cdots l_s$, where $l_i$ is a general linear form of $K[X_{i,1},\dots,X_{i,n_i}].$
\item[(C)] Assume that $\Char(K)=0$ and that $R$ is the homomorphic image of the Veronese ring:
$S=K[X_1^{a_1}\cdots X_s^{a_s} \vert \sum _{i=1} ^s a_i=b \text { and } a_i\geq 0].$ Then the $r$ linear forms can be chosen to be the images of $l^b$, where $l$ is a general linear form of $K[X_1,\dots,X_s].$
\item [(D)] Assume that $\Char(K)=0$ and that $R$ is the homomorphic image of Segre products of Veronese rings:
\[
S=K\left [ \prod_{\substack{1\leq i\leq s \\ 1\leq j \leq n_i}} X_{i,j}^{a_{i,j}} \text{ such that } \sum _{j} a_{i,j}=b_i \text { and } a_{i,j}\geq 0 \right ].
\] 
Then the $r$ linear forms can be chosen to be the images of $l_1^{b_1}\cdots l_s^{b_s}$, where $l_i$ is a general linear form of $K[X_{i,1},\dots,X_{i,n_i}].$

\item [(E)] Assume that $\Char(K)=0$ and that $R$ is the homomorphic image of the following toric ring: 
\[
S=K[X_{i_1} \cdots X_{i_s} \vert 1\leq i_1\leq n_1,\dots, 1\leq i_s\leq n_s \text{ and }n_1\leq n_2 \leq \dots \leq n_{s}].
\]   Then the $r$ linear forms can be chosen to be the images of $l_1(l_1+l_2)\cdots (l_1+l_2+\dots+l_s)$, where $l_i$ is a general linear form of $K[X_{n_{i-1},}\dots,X_{n_i}].$
\end{itemize}
\begin{proof} In all of the above cases it is straightforward to verify that the $r$ linear forms generate the homogeneous maximal ideal of $R$; this is the step that forces assumptions on the characteristic of $K$.
To conclude the proof it is enough show that for every sequence $\mathbf o=c_1,\dots, c_i$ with $0\leq i\leq r$  
the Hilbert function of $I_{\mathbf o}$ is well defined (i.e. constant on a Zariski open set) and independent of the order of the linear forms. We can follow the same outline of the proof of Proposition \ref{constant} and Corollary \ref{order}. All cases are quite similar and for the sake of the exposition we only show (C).

Write $R$ as $S/I$ where $S$ is the Veronese ring  and  $S\cong K[Z_1,\dots, Z_{\binom{b+s-1}{s-1}}]/J$. Let $\mathbb K$ be the fraction field of the polynomial  ring $K[y_{1,1}\dots,y_{1,s}, \dots, y_{r,1}\dots,y_{r,s}].$  Let $\mathbf{{\tilde l}}$ be the sequence of linear forms of the ring $S\otimes_K \mathbb K$ defined as  $\tilde{l_1}^b,\dots,\tilde{l_r}^b$ where, for $j=1,\dots,r$  we have set $\tilde l_j=y_{j,1}X_1+\cdots+y_{j,s}X_s.$  Fix a permutation $\sigma \in  \mathcal S_r$  and a sequence of operations $\mathbf o=c_1,\dots, c_i$ for some $i$, $0\leq i\leq r.$  We can compute algorithmically, by standard methods, a reduced Gr\"obner basis for the pre-image in $\mathbb K [Z_1,\dots, Z_{\binom{b+s-1}{s-1}}]$ of the ideal  $I_{\sigma \mathbf l, \mathbf o} \subset R\otimes_K \mathbb K$ where $\sigma \mathbf l$ is the image of $\sigma \mathbf{{\tilde l}}$ in $R \otimes_K \mathbb K.$ Let  $\alpha_1,\dots,\alpha_q$  be all  the non-zero coefficients of all the monomials in all the polynomials of  $\mathbb K [Z_1,\dots, Z_{\binom{b+s-1}{s-1}}]$  involved in such computation. 
The field $K$ is infinite and $\mathbb A(K^{sr})$ is irreducible, so we let $U_i$ be the non-empty Zariski open set of $\mathbb A(K^{sr})$ where the rational function $\alpha_i$ are defined and do not vanish. By setting $U_{\sigma,\mathbf o}=\cap_i^q U_i$ we get a non-empty Zariski open set of $\mathbb A(K^{sr})$  where the algorithm runs, roughly speaking, in a identical manner returning the same initial ideal no matter what point of $U_{\sigma, \mathbf o}$ is chosen as sequence of $r$ linear forms. To conclude we let $U$, the desired Zariski open set of $r$ general linear forms, to be the intersection of all the $U_{\sigma,\mathbf o}$'s. 
\end{proof}

\end{example}
\begin{remark}
Note that the characteristic assumption in the above (C),(D) and (E) is essential.  For instance, let $R=K[X_1^2,X_1X_2,X_2^2]/(X_1^2,X_2^2)\cong K[Z_1,Z_2,Z_3]/(Z_1,Z_3,Z_2^2-Z_1Z_3)$ and assume $\Char(K)=2.$ This correspond to the case $s=2$ and $b=2$ of example (C). The square of a general linear form of $K[X_1,X_2]$ can be written as $X_1^2+\lambda X_2^2$ and it has a zero image in $R.$ Property (2) of \textbf{(Gr,$d$)} is not satisfied. Moreover, a zero linear form clearly does not satisfy Green's estimate.
\end{remark}

We can now prove the desired hyperplane restriction theorem. Our proof follows the same outline of the classical result of Green \cite{Gr2}.
\begin{theorem}\label{NewGreenTHM} Let $R$ be a standard graded algebra and let $l_1,\dots,l_r$ be linear forms satisfying \textbf{(Gr,$d$)}. Then
\[
\dim_K (R/(l_1))_d \leq (\dim_K(R_d))_{\langle d \rangle }.
\] 
\end{theorem}
\begin{proof} 
We let ${\bf I}_{i,j}$ to be the set of all ideals $I_{\bf l, \bf o}$ with ${\bf l} =l_1,\dots, l_r$ and such that $\vert {\bf o} \vert = i\leq r $ and $\vert {\bf o} \vert_c=j.$ It is sufficient to prove the following.
\begin{claim*}
For every  $I\in {{\bf I}_{i,j}}$ with $0\leq i<r$ and $0\leq j<d$ we have
\begin{equation}\label{NewGreen} 
\dim_K(R/(I+(l_{i+1})))_{d-j} \leq (\dim_K(R/I)_{d-j})_{\langle d-j \rangle }.
\end{equation}
\end{claim*}
\noindent First of all, we show that the claim holds for all the ideals in ${{\bf I}_{i,j}}$ provided $i=r-1$ or $j=d-1.$ By part \eqref{22} of \textbf{(Gr,$d$)} when $I\in {{\bf I}_{r-1,j}}$  and $j<d$ by assumption, we get ${\bf m}\subseteq I+(l_r)$. Hence $(R/(I+(l_r)))_{d-j}=0$ and the inequality \eqref{NewGreen} holds.  If $I\in {{\bf I}_{i,j}}$ and $j= d-1$ the inequality \eqref{NewGreen} becomes  trivial when  ${\bf m }\subseteq I$ and otherwise it becomes $\dim_K(R/(I+(l_{i+1})))_{1} \leq \dim_K(R/I)_1-1$ which  follows from part \eqref{11} of \textbf{(Gr,$d$)}. 

We do a decreasing induction on the double index of ${{\bf I}_{i,j}}.$
 Let $I\in {{\bf I}_{i,j}}$ with $i<r-1$ and $j<d-1.$ 
By induction we know that  \eqref{NewGreen} holds for $(I+(l_{i+1}))\in  {{\bf I}_{i+1,j}}$ and for $(I:l_{i+1})\in  {{\bf I}_{i+1,j+1}}.$ 

Consider the sequence:
$
0 \rightarrow \frac{R}{(I+(l_{i+1})):l_{i+2}}(-1) \xrightarrow{\cdot l_{i+2}} \frac{R}{I+(l_{i+1})} \rightarrow \frac{R}{I+(l_{a+1})+(l_{a+2})} \rightarrow 0.
$ Let $\bar d=d-j.$ 

By looking at the graded component of degree $\bar d$ we get:
$
\notag \dim_K\left(\frac{R}{I+(l_{i+1})}\right )_{\bar d} = 
\dim_K \left( \frac{R}{(I+(l_{i+1})):l_{i+2}}\right )_{\bar d-1} 
 +\dim_K\left( \frac{R}{I+(l_{i+1})+(l_{i+2})}\right )_{\bar d}.
$

Property \eqref{33} of (Gr,$d$) implies $
\dim_K \left( \frac{R}{(I+(l_{i+1})):l_{i+2}}\right )_{\bar d-1}\leq \dim_K \left( \frac{R}{(I:l_{i+1})+(l_{i+2})}\right )_{\bar d-1},
$
and by using the inductive assumption  on $I:l_{i+1}$ and on $I+(l_{i+1})$
we deduce that \[\dim_K\left(\frac{R}{I+(l_{i+1})}\right )_{\bar d} \leq \left(\dim_K \left( \frac{R}{I:l_{i+1}}\right )_{\bar d-1}\right )_{\langle  \bar d-1 \rangle }+\left(\dim_K \left(\frac{R}{I+(l_{i+1})}\right )_{\bar d}\right)_{\langle \bar d \rangle }.\]

\vspace{0.2cm}
To simplify the notation, set $c=\dim_K(R/I)_{\bar d}$ and $c_H=\dim_K(R/(I+(l_{i+1})))_{\bar d}.$ From the short exact sequence $ 0 \rightarrow  \frac{R}{I:l_{i+1}}(-1) \xrightarrow {l_{i+1}} \frac{R}{I} \rightarrow \frac{R}{I+(l_{i+1})} \rightarrow 0$  we know that  $\dim_K \left( \frac{R}{I:l_{i+1}}\right )_{\bar d-1}=c-c_H,$ therefore the above upper bound for $\dim_K\left(\frac{R}{I+(l_{i+1})}\right )_{\bar d}$ becomes: 
\begin{equation}\label{GreenProof}
c_H\leq (c_H)_{\langle \bar d \rangle }+(c-c_H)_{\langle \bar d -1 \rangle }.
\end{equation}
Write $c_H=\binom{k_{\bar d}}{\bar d}+\binom{k_{\bar d-1}}{\bar d-1}+\dots + \binom{k_{\delta}}{\delta}.$  The inequality of the claim, which is $c_H\leq c_{\langle \bar d \rangle }$, is trivial when $c_H=0$, hence we can assume $k_{\delta}\geq \delta.$  Furthermore notice that  $c_H\leq c_{\langle \bar d \rangle }$ is equivalent to  $c\geq  \binom{k_{\bar d}+1}{\bar d}+\binom{k_{\bar d-1}}{\bar d-1}+\dots + \binom{k_{\delta}+1}{\delta}.$  If the claim fails we have:
\begin{equation}\label {Greenpp}
c-c_H< \binom{k_{\bar d}}{\bar d-1}+\binom{k_{\bar d-1}}{\bar d-2}+\dots + \binom{k_{\delta}}{\delta-1}.
\end{equation}
 We use \eqref{GreenProof} to derive a contradiction. There are two cases to consider.

If $\delta=1$ then \eqref{Greenpp} becomes $c-c_H\leq \binom{k_{\bar d}}{\bar d-1}+\binom{k_{\bar d-1}}{\bar d-2}+\dots + \binom{k_{2}}{1}.$ 

\noindent Thus  
\begin{alignat}{10}
(c-c_H)_{\langle \bar d-1 \rangle }&\leq \binom{k_{\bar d}-1}{\bar d-1}+\binom{k_{\bar d-1}-1}{\bar d-2}+\dots + \binom{k_{2}-1}{1} \quad& \text{and}\notag \\
(c_H)_{\langle \bar d \rangle }&\leq  \binom{k_{\bar d}-1}{\bar d}+\binom{k_{\bar d-1}-1}{\bar d-1}+\dots +  \binom{k_{2}-1}{2}+    \binom{k_{1}-1}{1}.\notag
\end{alignat} 
By adding these two inequalities, \eqref{GreenProof} gives 
\[
c_H\leq  \binom{k_{\bar d}}{\bar d}+\binom{k_{\bar d-1}}{\bar d-1}+\dots +\binom{k_2}{2} +   \binom{k_{1}-1}{1}<c_H,
\] 
which is a contradiction.

When  $\delta>1$ we can apply the function  $(-) _{\langle \bar d-1 \rangle }$   to both sides of \eqref{Greenpp}.  Since $k_{\delta}>\delta-1$ the last term of the right hand side stays positive, and the strict inequality is preserved. We get
\[ (c-c_H)_{\langle \bar d-1 \rangle }< \binom{k_{\bar d}-1}{\bar d-1}+\binom{k_{\bar d-1}-1}{\bar d-2}+\dots + \binom{k_{\delta}-1}{\delta -1}.
\]
By adding the last inequality to  $(c_H)_{\langle \bar d \rangle } \leq  \binom{k_{\bar d}-1}{\bar d}+\binom{k_{\bar d-1}-1}{\bar d-1}+\dots + \binom{k_{\delta}-1}{\delta}$ we obtain the following contradiction
\[ c_H< \binom{k_{\bar d}}{\bar d}+\binom{k_{\bar d-1}}{\bar d-1}+\dots + \binom{k_{\delta}}{\delta}=c_H.
\]
\end{proof}
A direct consequence of Theorem \ref{NewGreenTHM} is the corollary below.
\begin{corollary}\label{corollario Green}
Let $R$ be a standard graded algebra and let $l_1,\dots,l_r$ be linear forms satisfying \textbf{(Gr,$d$)}, and let the Macaulay representation of $\dim_K(R_d)$ be $\binom{k_d}{d}+\binom{k_{d-1}}{d-1}+\dots + \binom{k_{1}}{1}.$ Then for any $p$ such that $1\leq p\leq r$ we have
\[
\dim_K (R/(l_1,\dots, l_p))_d \leq \binom{k_d-p}{d}+\binom{k_{d-1}-p}{d-1}+\dots + \binom{k_{1}-p}{1}.
\] 
\end{corollary}
\begin{proof}By Theorem \ref{NewGreenTHM} we have $\dim_K (R/(l_1))_d \leq \binom{k_d-1}{d}+\binom{k_{d-1}-1}{d-1}+\dots + \binom{k_{1}-1}{1}.$ Note that the images of $l_2,\dots,l_r$ satisfy \textbf{(Gr,$d$)} for $R/(l_1)$. Thus we apply Theorem \ref{NewGreenTHM} and obtain the result by induction.
\end{proof}

We combine Corollary \ref{corollario Green}, Corollary \ref{order} and Remark \ref{irreducible} in the following result.

\begin{theorem}\label{MAIN1} Let $R=K[X_1,\dots,X_n]/I$ be a standard graded algebra and let $V\subset \mathbb{A}(R_1)$ be an irreducible variety spanning  $\mathbb{A}(R_1).$ 
Assume that $K$ is an algebraically closed or that the defining ideal of $V$ is homogeneous.
Then for any $p$ linear forms of $R$ that are general points of $V$ we have:
\[
\dim_K (R/(l_1,\dots, l_p))_d \leq \binom{k_d-p}{d}+\binom{k_{d-1}-p}{d-1}+\dots + \binom{k_{1}-p}{1}.
\] 
\end{theorem}
We are now ready to prove our main theorem, which is a description of the open set where Green's estimate holds. It is important to note that we make no assumption on the field $K.$
\begin{theorem}[Hyperplane Restriction]\label{MAIN}
Let $R$ be a standard graded algebra over a field $K$, and let $\mathbb{A}(R_1)$ be the affine space of  linear forms of $R_1.$
There exist finitely many proper linear subspaces $L_1,\dots,L_m$ of   $\mathbb{A}(R_1)$ such that for any form $l\not \in (\cup L
_i)$ we have:
\[
\dim_k (R/lR)_d\leq (\dim_K R_d)_{\langle d \rangle }. \label{Green's bound}
\]
\end{theorem}
\begin{proof} When $K$ is finite the result is trivial because $\mathbb{A}(R_1)$ is itself a finite union of proper linear subspaces. We assume that $K$ is infinite.
Note that the set of all linear forms not satisfying the above inequality is a Zariski closed set, say $V(I)$, of  $\mathbb {A}(R_1)$ which is defined by the vanishing of a certain ideal of minors $I\subseteq {\rm Sym} (R_1)$. Notice that $I$ is a homogeneous ideal. Assume by contradiction that $V(I)$ is not contained in any finite union of proper linear subspaces, hence there exists a irreducible component, say $V(P)$, spanning $\mathbb{A}(R_1)$. 
The prime ideal $P$ is homogeneous and by Theorem \ref{MAIN1} a general point in $V(P)$ satisfies Green's estimate, which is absurd.
\end{proof}

\section{Variations of the Eakin-Sathaye Theorem}
 
By using the rings described in Example \ref{examples of (Gr,$d$)} we can now prove a few versions of the  Eakin-Sathaye theorem. 

\begin{theorem}
 Let $(A,m)$ be a local ring with infinite residue field $K$. Let $I$ be an ideal of $A$. Let $i$ and $p$ be
positive integers. If the number of minimal generators of
$I^i$, denoted by $v(I^i)$, satisfies $v(I^i)<\binom{i+p}{i}$ then 
\begin{itemize}
\item[(a)] (Eakin-Sathaye) There are elements $h_1,\dots,h_p$ in $I$ such that
$I^i=(h_1,\dots,h_p)I^{i-1}.$
\end{itemize}
Moreover: 
\begin{itemize}
\item [(b)](O'Carroll) If $I=I_1 \cdots I_s$, where $I_j$'s are ideals of $A$, the elements  $h_j$'s can be chosen of the form $l_1\cdots l_s$ with $l_i\in I_i.$
\item[(c)] Assume $\Char(K)=0$. If $I=J^b$ , where $J$ is an ideal of $A$, the elements $h_j$'s can be chosen of the form $l^b$ with $l\in J.$
\item[(d)] Assume $\Char(K)=0$. If $I=I_1^{b_1} \cdots I_s^{b_s}$, where $I_j$'s are ideals of $A$, the elements  $h_j$ can be chosen of the form $l_1^{b_1}\cdots l_s^{b_s}$ with $l_i\in I_i.$
\item[(e)] Assume $\Char(K)=0$. If $I=I_1(I_1+I_2)\cdots (I_1+\dots +I_s)$, where $I_j$'s are ideals of $A$, the elements  $h_j$ can be chosen of the form $l_1(l_1+l_2)\cdots (l_1+\dots +l_s)$ with $l_i\in I_i.$ 
\end{itemize}
\end{theorem}
\begin{proof}
 First of all, note that  since
$v(I^i)$ is finite, without loss of generality we can assume that $I$
is also finitely generated: in fact, if $H\subseteq I$ is a finitely
generated ideal such that $H^i=I^i$ the result for $H$ implies
the one for $I.$ Similarly, we can also assume that the ideals $I_j$ of (b),(d) and (e) and the ideal $J$ of (c) are finitely generated.
By the use of Nakayama's Lemma, we can replace $I$ by the
homogeneous maximal ideal of the fiber cone $R=\bigoplus_{i\geq 0} I^i/mI^i.$
 Note that $R$ is a standard graded algebra finitely generated over the infinite field $R/m=K.$ Moreover, the algebras $R$ of (a),(b),(c),(d),(e) satisfy the properties of the Example \ref{examples of (Gr,$d$)} parts (A),(B),(C),(D), and (E) respectively. Let $l_1,\dots, l_r$ as in Example \ref{examples of (Gr,$d$)} and assume also that $p\leq r.$ The theorem is proved if we can show that $(R/(l_1,\dots,l_p))_i=0.$ Note that $\dim_K R_i  \leq \binom{i+p}{i}-1= \binom{i+p-1}{i}+\binom{i+p-2}{i-1}+\dots+\binom{i+p-j}{i-j+1}+\dots
+\binom{p}{1}.$ The last equality can be proved directly or alternatively one can
order the arrays of Macaulay coefficients by using the
lexicographic order and observe that
$(i+p,0,\dots,0)$ is preceded by
$(i+p-1,i+p-2,\dots,p).$ By Corollary \ref{corollario Green} we deduce
\[
\dim_K (R/(l_1,\dots,l_p))_i\leq
\binom{i-1}{i}+\binom{i-2}{i-1}+\dots
+\binom{0}{1}.
\]
The term on the right hand side is zero and therefore the theorem is proved.
\end{proof}

\end{document}